\documentclass[oneside,english]{amsart}
\usepackage[T1]{fontenc}
\usepackage[utf8]{inputenc}
\usepackage{babel}
\usepackage{amstext}
\usepackage{amsthm}
\usepackage{amssymb}
\usepackage[unicode=true,pdfusetitle,
 bookmarks=true,bookmarksnumbered=false,bookmarksopen=false,
 breaklinks=false,pdfborder={0 0 1},backref=false,colorlinks=false]
 {hyperref}

\makeatletter


\numberwithin{equation}{section}
\numberwithin{figure}{section}
\theoremstyle{plain}
\newtheorem{thm}{\protect\theoremname}[section]
  \theoremstyle{plain}
  \newtheorem{cor}[thm]{\protect\corollaryname}
  \theoremstyle{plain}
  \newtheorem{prop}[thm]{\protect\propositionname}
  \theoremstyle{plain}
  \newtheorem{lem}[thm]{\protect\lemmaname}
  \theoremstyle{remark}
  \newtheorem{rem}[thm]{\protect\remarkname}

\makeatother

  \providecommand{\corollaryname}{Corollary}
  \providecommand{\lemmaname}{Lemma}
  \providecommand{\propositionname}{Proposition}
  \providecommand{\remarkname}{Remark}
\providecommand{\theoremname}{Theorem}

\begin{document}

\title{Discrepancies of $p$-cyclic quotient varieties}

\author{Takehiko Yasuda}
\begin{abstract}
We consider the quotient variety associated to a linear representation
of the cyclic group of order $p$ in characteristic $p>0$. We estimate
the minimal discrepancy of exceptional divisors over the singular
locus. In particular, we give criteria for the quotient variety being
terminal, canonical and log canonical. As an application, we obtain
new examples of non-Cohen-Macaulay terminal singularities, adding
to examples recently announced by Totaro. 
\end{abstract}

\address{Department of Mathematics, Graduate School of Science, Osaka University,
Toyonaka, Osaka 560-0043, Japan, tel:+81-6-6850-5326, fax:+81-6-6850-5327}

\curraddr{Mathematical Institute, Tohoku University, Aoba, Sendai, 980-8578,
JAPAN}

\email{takehiko.yasuda.a5@tohoku.ac.jp}

\thanks{This work was supported by JSPS KAKENHI Grant Numbers JP15K17510
and JP16H06337. }

\maketitle
\global\long\def\bigmid{\mathrel{}\middle|\mathrel{}}

\global\long\def\AA{\mathbb{A}}
\global\long\def\PP{\mathbb{P}}
\global\long\def\NN{\mathbb{N}}
\global\long\def\GG{\mathbb{G}}
\global\long\def\ZZ{\mathbb{Z}}
\global\long\def\QQ{\mathbb{Q}}
\global\long\def\CC{\mathbb{C}}
\global\long\def\FF{\mathbb{F}}
\global\long\def\LL{\mathbb{L}}
\global\long\def\RR{\mathbb{R}}
\global\long\def\MM{\mathbb{M}}
\global\long\def\SS{\mathbb{S}}

\global\long\def\bx{\boldsymbol{x}}
\global\long\def\by{\boldsymbol{y}}
\global\long\def\bf{\mathbf{f}}
\global\long\def\ba{\mathbf{a}}
\global\long\def\bs{\mathbf{s}}
\global\long\def\bt{\mathbf{t}}
\global\long\def\bw{\mathbf{w}}
\global\long\def\bb{\mathbf{b}}
\global\long\def\bv{\mathbf{v}}
\global\long\def\bp{\mathbf{p}}
\global\long\def\bq{\mathbf{q}}
\global\long\def\bm{\mathbf{m}}
\global\long\def\bj{\mathbf{j}}
\global\long\def\bM{\mathbf{M}}
\global\long\def\bd{\mathbf{d}}
\global\long\def\bz{\mathbf{z}}
\global\long\def\bG{\mathbf{G}}
\global\long\def\bg{\mathbf{g}}
\global\long\def\bh{\mathbf{h}}

\global\long\def\cN{\mathcal{N}}
\global\long\def\cW{\mathcal{W}}
\global\long\def\cY{\mathcal{Y}}
\global\long\def\cM{\mathcal{M}}
\global\long\def\cF{\mathcal{F}}
\global\long\def\cX{\mathcal{X}}
\global\long\def\cE{\mathcal{E}}
\global\long\def\cJ{\mathcal{J}}
\global\long\def\cO{\mathcal{O}}
\global\long\def\cD{\mathcal{D}}
\global\long\def\cZ{\mathcal{Z}}
\global\long\def\cR{\mathcal{R}}
\global\long\def\cC{\mathcal{C}}
\global\long\def\cL{\mathcal{L}}
\global\long\def\cV{\mathcal{V}}
\global\long\def\cI{\mathcal{I}}
\global\long\def\cH{\mathcal{H}}
\global\long\def\cK{\mathcal{K}}
\global\long\def\cS{\mathcal{S}}

\global\long\def\fs{\mathfrak{s}}
\global\long\def\fp{\mathfrak{p}}
\global\long\def\fm{\mathfrak{m}}
\global\long\def\fX{\mathfrak{X}}
\global\long\def\fV{\mathfrak{V}}
\global\long\def\fx{\mathfrak{x}}
\global\long\def\fv{\mathfrak{v}}
\global\long\def\fY{\mathfrak{Y}}
\global\long\def\fa{\mathfrak{a}}
\global\long\def\fb{\mathfrak{b}}
\global\long\def\fc{\mathfrak{c}}
\global\long\def\fD{\mathfrak{D}}
\global\long\def\fS{\mathfrak{S}}
\global\long\def\fO{\mathfrak{O}}
\global\long\def\fM{\mathfrak{M}}

\global\long\def\rv{\mathbf{\mathrm{v}}}
\global\long\def\rx{\mathrm{x}}
\global\long\def\rw{\mathrm{w}}
\global\long\def\ry{\mathrm{y}}
\global\long\def\rz{\mathrm{z}}
\global\long\def\bv{\mathbf{v}}
\global\long\def\bw{\mathbf{w}}
\global\long\def\sv{\mathsf{v}}
\global\long\def\sx{\mathsf{x}}
\global\long\def\sw{\mathsf{w}}

\global\long\def\Spec{\operatorname{Spec}}
\global\long\def\id{\mathrm{id}}

\global\long\def\Gal{\operatorname{Gal}}
\global\long\def\det{\mathrm{det}}
\global\long\def\Aut{\mathrm{Aut}}
\global\long\def\sm{\mathrm{sm}}
\global\long\def\sing{\mathrm{sing}}
\global\long\def\discrep#1{\mathrm{discrep}\left(#1\right)}
\global\long\def\ord{\mathrm{ord}\,}
\global\long\def\Exc#1{\mathrm{Exc}\left(#1\right)}
\global\long\def\mld#1#2{\mathrm{mld}(#1;#2)}
\global\long\def\center{\mathrm{center}}
\global\long\def\sht{\mathrm{sht}}
\global\long\def\st{\mathrm{st}}
\global\long\def\ae{\mathrm{a.e.}}

\section{Introduction}

Let $k$ be an algebraically closed field of characteristic $p>0$
and $G=\ZZ/p$ the cyclic group of order $p$. Suppose that $G$ linearly
acts on the $d$-dimensional affine space $V=\AA_{k}^{d}$. Let $X:=V/G$
be the quotient variety. This variety is factorial (see \cite[Th. 3.8.1]{MR2759466}),
but not necessarily Cohen-Macaulay. We are interested in singularities
of $X$ from the viewpoint of the minimal model program. 

To state our main results, we recall basic notions concerning singularities
and introduce some notation. Let us consider a modification (proper
birational morphism) $f\colon Y\to X$ such that $Y$ is normal, and
the exceptional locus $\Exc f\subset Y$ and the preimage $f^{-1}(X_{\sing})$
of $X_{\sing}$ are both of pure dimension $d-1$. We will call such
a morphism an \emph{admissible modification. }Note that the last condition
implies $f^{-1}(X_{\sing})\subset\Exc f$. Note also that an arbitrary
modification $Y\to X$ can be altered into an admissible one by blowup
and normalization. For an admissible modification $f\colon Y\to X$,
we define the \emph{relative canonical divisor} $K_{Y/X}$ in the
usual way, which is a Weil divisor with support contained in $\Exc f$.
Let $\Exc f=\bigcup_{i\in\cE_{f}}E_{i}$ and $f^{-1}(X_{\sing})=\bigcup_{i\in\cS_{f}}E_{i}$
be the decompositions into irreducible components with $\cS_{f}\subset\cE_{f}$
and write $K_{Y/X}=\sum_{i\in\cE_{f}}a_{i}E_{i}$, where $a_{i}$
are integers called \emph{discrepancies}. We define
\begin{align*}
\delta(X):=\discrep{\center\subset X_{\sing};X} & =\inf_{f}\min_{i\in\cS_{f}}a_{i}.
\end{align*}
Here $f$ runs over admissible modifications of $X$. We have either
$\delta(X)\ge-1$ or $\delta(X)=-\infty$ \cite[Cor. 2.31]{MR1658959}.
We say that $X$ is \emph{terminal} (resp. \emph{canonical, log canonical})
if $\delta(X)>0$ (resp. $\ge0$, $\ge-1$). Note that since our variety
$X$ is factorial and hence $\delta(X)$ is an integer if not $-\infty$,
that $X$ is canonical is equivalent to that $X$ is log terminal
(meaning $\delta(X)>-1$). 

We will estimate $\delta(X)$ in terms of the given $G$-representation
$V$. For each integer $i$ with $1\le i\le p$, there exists a unique
indecomposable representation of $G$ over $k$; we denote it as $V_{i}$.
The representation $V$ decomposes into indecomposable ones: $V=\bigoplus_{\lambda=1}^{l}V_{d_{\lambda}}$
with $1\le d_{\lambda}\le p$ and $\sum d_{\lambda}=d$. The decomposition
is unique up to permutation of direct summands. We define an invariant
$D_{V}$ by
\[
D_{V}:=\sum_{\lambda=1}^{l}\frac{(d_{\lambda}-1)d_{\lambda}}{2}.
\]
We easily see that $D_{V}=0$ if and only if the $G$-action is trivial,
and also that $D_{V}=1$ if and only if a generator of $G$ is a pseudo-reflection
(that is, the fixed point locus $V^{G}$ has codimension one). In
these cases, the quotient variety $X$ is again isomorphic to $\AA_{k}^{d}$.
Excluding these cases, we assume in what follows that $D_{V}\ge2$.
The fixed point locus $V^{G}$ is then an $l$-dimensional linear
subspace of $V$. The quotient morphism $V\to X$ is \'etale outside
$V^{G}$ and the image of $V^{G}$ is exactly the singular locus $X_{\sing}$
of $X$. The invariant $D_{V}$ is clearly additive with respect to
direct sums. For indecomposable representations of small dimensions,
we have
\[
D_{V_{1}}=0,\,D_{V_{2}}=1,\,D_{V_{3}}=3,\,D_{V_{4}}=6,\,D_{V_{5}}=10.
\]
From \cite{MR581583}, $X$ is Cohen-Macaulay if and only if $d-l=\mathrm{codim}\,V^{G}\le2$.
If $D_{V}=2$, then $V=V_{2}^{\oplus2}\oplus V_{1}^{\oplus(d-4)}$
for some $n$ and $X$ is Cohen-Macaulay. Note that $V_{1}^{\oplus n}$
is the $n$-dimensional trivial representation. If $D_{V}=3$, then
$V$ is either $V_{3}\oplus V_{1}^{\oplus(d-3)}$ or $V_{2}^{\oplus3}\oplus V_{1}^{\oplus(d-6)}$;
$X$ is Cohen-Macaulay in the former and not Cohen-Macaulay in the
latter. If $D_{V}\ge4$, then $X$ is never Cohen-Macaulay. 

In \cite[Propositions 6.6 and 6.9]{MR3230848}, the author proved
the following result.
\begin{thm}
\label{thm:old}If $D_{V}\ge p$, then $X$ is canonical. 
\end{thm}

This result provided the first example of log terminal (even canonical)
but not Cohen-Macaulay singularities in all positive characteristics
(recall that log terminal singularities in characteristic zero are
always Cohen-Macaulay). For instance, $X$ has such singularities
when $V=V_{2}^{\oplus3}$ in characteristic two or $V=V_{4}$ in characteristic
five. Later, further examples of non-Cohen-Macaulay log terminal or
canonical singularities \cite{Gongyo:2015kl,Cascini:2016lr,Kovacs:2017cl,Bernasconi:2017jl}
were found. The above theorem also shows that when $V=V_{3}$ in characteristic
$p\ge5$, $X$ is Cohen-Macaulay (in fact, a hypersurface) but not
log terminal (recall that quotient singularities in characteristic
zero are always log terminal). It should be mentioned that Hacon and
Witaszek \cite{Hacon:2017xw} proved that in sufficiently large characteristics,
three-dimensional log terminal singularities are Cohen-Macaulay. 

The aim of this paper is to strengthen the above theorem. For a positive
integer $j$ with $p\nmid j$, we define 
\[
\sht_{V}(j):=\sum_{\lambda=1}^{l}\sum_{i=1}^{d_{\lambda}-1}\left\lfloor \frac{ij}{p}\right\rfloor .
\]
Here $\left\lfloor \cdot\right\rfloor $ denotes the round down of
rational numbers to integers. The following two theorems are our main
results:
\begin{thm}
\label{thm:precise}Suppose that $D_{V}\ge2$. Then $D_{V}<p-1$ if
and only if $\delta(X)=-\infty$. If $D_{V}\ge p-1$, then 
\begin{align}
\delta(X) & =d-1-l-\max_{1\le s\le p-1}\{s-\sht_{V}(s)\}\label{eq:intro-1}\\
 & =D_{V}-1-\max_{1\le s\le p-1}\{\sht_{V}(p-s)+s\}.\label{eq:intro-2}
\end{align}
\end{thm}

\begin{thm}
\label{thm:log-term}Suppose that $D_{V}\ge2$. Then 
\begin{equation}
\delta(X)\le D_{V}-p.\label{eq:intro-3}
\end{equation}
If $D_{V}\ge p$, then we also have
\begin{equation}
\delta(X)\ge\frac{2D_{V}}{p}-2.\label{eq:intro-4}
\end{equation}
\end{thm}

As a direct consequence of these, we obtain:
\begin{cor}
Suppose that $D_{V}\ge2$. Then $X$ is terminal (resp. canonical,
log canonical) if and only if $D_{V}>p$ (resp. $D_{V}\ge p$, $D_{V}\ge p-1$). 
\end{cor}

For instances, in the cases where $V=V_{2}^{\oplus3}$ in characteristic
two and $V=V_{4}$ in characteristic five, the quotient variety $X$
is terminal but not Cohen-Macaulay. The dimension four is the smallest
possible, because we consider linear actions and the fixed-point locus
$V^{G}$ is always of positive dimension. Note that Totaro \cite[Cor. 2.2]{Totaro:2017bk}
had earlier announced construction of terminal but not Cohen-Macaulay
singularities by using homogeneous varieties with nonreduced stabilizers.
These singularities have dimension $\ge8$. The current work was motivated
by his result. Later, inspired by construction of the author, Totaro
constructed an example of a three-dimensional terminal but not Cohen-Macaulay
singularity in characteristic two, by considering a non-linear action
of $\ZZ/2$ (see \cite[Theorem 5.1]{Totaro:2017bk}).

In the next section, we review basics of motivic integration and give
an expression of $\delta(X)$ in terms of motivic integration (Proposition
\ref{prop:precise-1}). In Section \ref{sec:Proof-of-Theorem}, we
prove Theorems \ref{thm:precise} and \ref{thm:log-term}.

The essential part of this work was done during a workshop held at
Mathematisches Forschungsinstitut Oberwolfach on September, 2017.
The author thanks the organizers of the workshop and the host institute.
He also thanks Yoshinori Gongyo, Shihoko Ishii, Sándor Kovács, Burt
Totaro and Shunsuke Takagi for helpful conversation and information.

\section{Motivic integration}

In this section, after briefly recalling basics of motivic integration,
we prove a result expressing $\delta(X)$ in terms of motivic integration,
Proposition \ref{prop:precise-1}, which will be necessary in Section
\ref{sec:Proof-of-Theorem}.

We denote by $X$ an irreducible variety of dimension $d$ over $k$.
We denote the arc space of $X$ by $J_{\infty}X$. This space has
the motivic measure denoted by $\mu_{X}$. We let the motivic measure
take values in the ring denoted by $\hat{\cM}'$ in \cite{MR3230848},
a version of the completed Grothendieck ring of varieties. The element
of $\hat{\cM}'$ defined by a variety $Z$ is denoted by $[Z]$. We
write the special element $[\AA_{k}^{1}]$ as $\LL$, which is invertible
by construction. There exists a ring homomorphism $P\colon\hat{\cM'}\to\ZZ((T^{-1}))$
which sends $[Z]$ to the Poincaré polynomial of $Z$ (see \emph{ibid.},
pp.\,1141-1142). We define the \emph{degree} of $\alpha=\sum a_{i}T^{i}\in\ZZ((T^{-1}))$
by $\deg\alpha:=\sup\{i\mid a_{i}\ne0\}$. For $a\in\hat{\cM}'$,
we define its \emph{dimension }as $\dim a:=\frac{1}{2}\deg P(a)$
so that for a variety $Z$ and an integer $n$, we have $\dim[Z]\LL^{n}=\dim Z+n$. 

For $n\in\NN$, let $\pi_{n}\colon J_{\infty}X\to J_{n}X$ be the
truncation map to $n$-jets. A subset $U\subset J_{\infty}X$ is called
\emph{stable }if there exists $n\in\NN$ such that $\pi_{n}(U)$ is
a constructible subset of $J_{n}X$, $U=\pi_{n}^{-1}\pi_{n}(U)$ and
for every $n'\ge n$, the map $\pi_{n'+1}(U)\to\pi_{n'}(U)$ is a
piecewise trivial $\AA^{d}$-fibration. The measure $\mu_{X}(U)$
of a stable subset $U$ is defined to be $[\pi_{n}(U)]\LL^{-nd}$
for $n\gg0$. A \emph{measurable subset} of $J_{\infty}X$ is a subset
approximated by a sequence of stable subsets. The measure of a measurable
subset is defined as the limit of ones of stable subsets (for details,
see \cite[Appendix]{MR1905024}, \cite[Section 6]{MR2075915}).

In what follows, we assume that $X$ is normal and the canonical sheaf
$\omega_{X}$ is invertible. The \emph{$\omega$-Jacobian ideal} $\cJ_{X}\subset\cO_{X}$
is then defined by $\cJ_{X}\omega_{X}=\mathrm{Im}\left(\bigwedge^{d}\Omega_{X/k}\to\omega_{X}\right).$
For an admissible modification $f\colon Y\to X$, let $f_{\infty}\colon J_{\infty}Y\to J_{\infty}X$
be the induced map of arc spaces. For a measurable subset $U\subset J_{\infty}Y_{\sm}\subset J_{\infty}Y$,
where $Y_{\sm}$ denotes the smooth locus, we have the \emph{change
of variables formula},
\[
\int_{U}\LL^{-\ord K_{Y/X}}\,d\mu_{Y}=\int_{f_{\infty}(U)}\LL^{\ord\cJ_{X}}\,d\mu_{X}.
\]
Here $\ord?$ denotes the order function associated to a divisor or
an ideal sheaf. The proofs of our main results are based on repeated
evaluation of dimensions of integrals as above. For this reason, we
introduce the following notation: For a measurable subset $U$ of
$J_{\infty}X$ (resp. $J_{\infty}Y_{\sm}$), we define
\[
\nu(U):=\dim\left(\int_{U}\LL^{\ord\cJ_{X}}\,d\mu_{X}\right)\quad\left(\text{resp. }\nu(U):=\dim\left(\int_{U}\LL^{-\ord K_{Y/X}}\,d\mu_{Y}\right)\right),
\]
provided that the integral converges. We say that a measurable subset
$U$ of $J_{\infty}X$ or $J_{\infty}Y_{\sm}$ is \emph{small }if
the relevant integral converges. When $U\subset J_{\infty}Y_{\sm}$,
we have $\nu(U)=\nu(f_{\infty}(U))$. 

In what follows, we write $A=_{\ae}B$ (resp. $A\subset_{\ae}B$)
to mean that there exists a measure zero subset $C$ such that $A\setminus C=B\setminus C$
(resp. $A\setminus C\subset B\setminus C$); ``a.e.'' stands for
``almost everywhere''. We also denote by $\pi_{X}$ the truncation
map $J_{\infty}X\to J_{0}X=X$ and similarly for $\pi_{Y}$. 
\begin{prop}
\label{prop:precise-1}Let $C_{r}\subset\pi_{X}^{-1}(X_{\sing})$,
$r\in\NN$ be a countable collection of small measurable subsets such
that $\pi_{X}^{-1}(X_{\sing})=_{\ae}\bigcup_{r\in\NN}C_{r}$. Then
\begin{align*}
\delta(X) & =d-1-\sup_{r}\nu(C_{r}).
\end{align*}
\end{prop}

The corresponding result in characteristic zero would be well-known
to specialists and is an easy consequence of the existence of log
resolution, the change of variables formula and explicit computation
of $\int_{U}\LL^{-\ord K_{Y/X}}$ for some small measurable subsets
$U\subset J_{\infty}Y$. Since we work in positive characteristic,
we do not know the existence of log resolution. However a result of
Reguera \cite[Prop. 3.7(vii)]{MR2503985}, that every stable point
determines a divisorial valuation, can take the place of the existence
of log resolution. Indeed similar results for MJ-discrepancies have
been proved by using it \cite[Th. 3.18]{MR3751297}. The rest of this
section is devoted to the proof of the above proposition, which uses
only standard arguments. We first prove a few auxiliary results.
\begin{lem}
\label{lem:ineq}Let $C,C_{i},D_{i}$, $i\in\NN$ be small measurable
subsets of $J_{\infty}X$.
\begin{enumerate}
\item If $C\subset_{\ae}\bigcup_{i\in\NN}D_{i}$, then $\nu(C)\le\sup_{i\in\NN}\nu(D_{i})$.
\item If $\bigcup_{i\in\NN}C_{i}=_{\ae}\bigcup_{i\in\NN}D_{i}$, then $\sup_{i\in\NN}\nu(C_{i})=\sup_{i\in\NN}\nu(D_{i}).$
\end{enumerate}
\end{lem}

\begin{proof}
For the first assertion, let $D_{i}':=D_{i}\setminus\bigcup_{j<i}D_{j}$.
Then the sets $D_{i}'$ are mutually disjoint and $\bigcup_{i}D_{i}'=\bigcup_{i}D_{i}$.
Therefore 
\[
\int_{C}\LL^{\ord\cJ_{X}}\,d\mu_{X}=\sum_{i\in\NN}\int_{C\cap D_{i}'}\LL^{\ord\cJ_{X}}\,d\mu_{X}.
\]
It follows that
\[
\nu(C)=\sup_{i\in\NN}\nu(C\cap D_{i}')\le\sup_{i\in\NN}\nu(D_{i}')\le\sup_{i\in\NN}\nu(D_{i}).
\]
Thus the first assertion holds. For the second assertion, we apply
the first assertion to $C=C_{i}$ for each $i$ to get
\[
\nu(C_{i})\le\sup_{j\in\NN}\nu(D_{j}).
\]
This shows $\sup_{i}\nu(C_{i})\le\sup_{i}\nu(D_{i})$. The opposite
inequality is proved by the same argument. 
\end{proof}
For $i\in\cE_{f}$ and an integer $b>0$, we define 
\begin{align*}
E_{i}^{\circ} & :=\left(E_{i,\sm}\cap Y_{\sm}\right)\setminus\bigcup_{j\in\cE_{f}\setminus\{i\}}E_{j}\subset Y,\\
N_{i,b} & :=\{\gamma\in\pi_{Y}^{-1}(E_{i}^{\circ})\mid\mathrm{ord}_{E_{i}}(\gamma)=b\}\subset J_{\infty}Y.
\end{align*}
To emphasize $f$, we also write $N_{i,b}$ as $N_{i,b}^{f}$. By
a standard computation of motivic integration, 
\[
\int_{N_{i,b}}\LL^{-\ord K_{Y/X}}\,d\mu_{Y}=\mu_{Y}(N_{i,b})\LL^{-a_{i}b}=[E_{i}^{\circ}](\LL-1)\LL^{-(1+a_{i})b}.
\]
In particular, 
\[
\nu(N_{i,b})=d-(1+a_{i})b.
\]
\begin{lem}
\label{lem:delta}Let $\fM$ be the set of (isomorphism classes of)
admissible modifications $f\colon Y\to X$. We have
\[
\delta(X)=d-1-\sup_{f\in\fM,i\in\cS_{f},b>0}\nu(N_{i,b}^{f}).
\]
\end{lem}

\begin{proof}
If $X$ is not log canonical, then the both sides are $-\infty$.
If $X$ is log canonical, since $1+a_{i}\ge0$, we have $\nu(N_{i,b}^{f})\le\nu(N_{i,1}^{f})$
for every $f\in\fM$, $i\in\cS_{f}$, $b>0$. Therefore
\begin{align*}
\delta(X) & =\inf_{f\in\fM,i\in\cS_{f}}a_{i}\\
 & =d-1-\sup_{f\in\fM,i\in\cS_{f}}\nu(N_{i,1}^{f})\\
 & =d-1-\sup_{f\in\fM,i\in\cS_{f},b>0}\nu(N_{i,b}^{f}).
\end{align*}
\end{proof}
\begin{lem}
\label{lem:almost-everywhere}There exists a countable set $\fM_{0}\subset\fM$
of admissible modifications of $X$ such that 
\[
\pi_{X}^{-1}(X_{\sing})=_{\ae}\bigcup_{f\in\fM_{0},\,i\in\cS_{f},\,b>0}f_{\infty}(N_{i,b}^{f}).
\]
\end{lem}

\begin{proof}
The lemma follows from the following two facts:
\begin{enumerate}
\item There exist countably many measurable subsets $C_{j}\subset\pi_{X}^{-1}(X_{\sing})$,
$j\in\NN$ such that $\pi_{X}^{-1}(X_{\sing})=_{\ae}\bigcup_{j}C_{j}$. 
\item Every measurable subset $C\subset\pi_{X}^{-1}(X_{\sing})$ is almost
everywhere covered by $f_{\infty}(\pi_{Y}^{-1}(E_{i}^{\circ}))$ for
countably many $f$'s. 
\end{enumerate}
For the first fact, we can for instance take $C_{j}$ to be the locus
of arcs having order $j+1$ along $X_{\sing}$ (\cite[Lemma 4.1]{MR1664700},
\cite[Lemme 4.5.4]{MR2075915}). For the second one, we need a result
of Reguera \cite[Prop. 3.7(vii)]{MR2503985} that every stable point
of $J_{\infty}X$ determines a divisorial valuation on the function
field, which means that for every irreducible stable subset $C\subset\pi_{X}^{-1}(X_{\sing})$,
there exist an admissible modification $f\colon Y\to X$ and $i\in\cS_{f}$
such that $f_{\infty}(\pi_{Y}^{-1}(E_{i}^{\circ}))$ contains the
generic point of $C$. Let us take stable subsets $D_{1},\dots,D_{m}\subset C$
such that 
\[
\dim\mu_{X}\left(C\setminus\bigcup_{h=1}^{m}D_{h}\right)<\dim\mu_{X}(C).
\]
Applying Reguera's result to the generic points of $D_{h}$, we can
make $\dim\mu_{X}(C)$ strictly smaller by removing finitely many
subsets of the form $f_{\infty}(\pi_{Y}^{-1}(E_{i}^{\circ}))$ from
$C$. By an inductive argument, we get a measure zero subset by removing
countably many subsets of the same form, equivalently $C$ is almost
everywhere covered by countably subsets of this form. Since each subset
of the form $f_{\infty}(\pi_{Y}^{-1}(E_{i}^{\circ}))$ is almost everywhere
covered by countably many subsets of the form $f_{\infty}(N_{i,b}^{f})$,
the second fact above holds. 
\end{proof}
\begin{proof}[Proof of Proposition \ref{prop:precise-1}]
For any admissible modification $f$ and $i\in\cS_{f}$, since $f_{\infty}(N_{i,1}^{f})\subset_{\ae}\bigcup_{r}C_{r}$,
from Lemma \ref{lem:ineq}, we have
\[
a_{i}=d-1-\nu(N_{i,1})\ge d-1-\sup_{r}\nu(C_{r}).
\]
Thus $\delta(X)\ge d-1-\sup_{r}\nu(C_{r})$. 

Let $\fM_{0}$ be a countable set of admissible modifications of $X$
as in Lemma \ref{lem:almost-everywhere}. Then 
\[
\bigcup_{r}C_{r}=_{\ae}\bigcup_{f\in\fM_{0},i\in\cS_{f},b>0}f_{\infty}(N_{i,b}^{f})
\]
and from Lemma \ref{lem:ineq},
\[
\sup_{r}\nu(C_{r})=\sup_{f\in\fM_{0},i\in\cS_{f},b>0}\nu(N_{i,b}^{f}).
\]
This equality together with Lemma \ref{lem:delta} shows
\begin{align*}
\delta(X) & =d-1-\sup_{f\in\fM,i\in\cS_{f},b>0}\nu(N_{i,b}^{f})\\
 & \le d-1-\sup_{f\in\fM_{0},i\in\cS_{f},b>0}\nu(N_{i,b}^{f})\\
 & =d-1-\sup_{r}\nu(C_{r}).
\end{align*}
\end{proof}

\section{Proof of Theorems \ref{thm:precise} and \ref{thm:log-term}\label{sec:Proof-of-Theorem}}

In this section, we prove Theorems \ref{thm:precise} and \ref{thm:log-term},
applying Proposition \ref{prop:precise-1} to a special choice of
small measurable subsets for which the values of $\nu$ were explicitly
computed in \cite{MR3230848}. 

Let $V=\AA_{k}^{d}$ and $X=V/G$ be as in Introduction. We assume
$D_{V}\ge2$. 
\begin{prop}
\label{prop:M_j}Let $J:=\{j\in\ZZ\mid j>0,\,p\nmid j\}\cup\{0\}$.
There exist measurable subsets $M_{j}\subset\pi_{X}^{-1}(X_{\sing})$,
$j\in J$ such that $\bigcup_{j\in J}M_{j}=_{\ae}\pi_{X}^{-1}(X_{\sing})$
and
\[
\int_{M_{j}}\LL^{\ord\cJ_{X}}\,d\mu_{X}=\begin{cases}
(\LL-1)\LL^{l+j-1-\lfloor j/p\rfloor-\sht_{V}(j)} & (j>0)\\
\LL^{l} & (j=0).
\end{cases}
\]
\end{prop}

\begin{proof}
This follows from computation in Proof of \cite[Prop. 6.9]{MR3230848}
and a version of the change of variables formula, \emph{ibid.}, Theorem
5.20. For $j>0$, we take $M_{j}$ to be the image of $\cJ_{\infty,j}\cX$,
the space of twisted arcs with ramification jump $j$ (for the definition,
see \emph{ibid.}, Definition 3.11). A point of $\cJ_{\infty,j}\cX$
corresponds to a $G$-equivariant morphism $\Spec\cO_{L}\to V$, where
$L$ is a Galois extension of $k((t))$ with Galois group $G$ and
ramification jump $j$, and $\cO_{L}$ is the integral closure of
$k[[t]]$ in $L$. The closed point of $\Spec\cO_{L}$ maps into $V^{G}$.
This implies $M_{j}\subset\pi_{X}^{-1}(X_{\sing})$. For $j=0$, we
take $M_{0}$ to be the intersection of the image of $\cJ_{\infty,0}\cX$
and $\pi_{X}^{-1}(X_{\sing})$. Since the map $\bigsqcup_{j\in J}\cJ_{\infty,j}\cX\to J_{\infty}X$
is almost bijective (\emph{ibid.}, Proposition 3.17), these sets $M_{j}$
almost cover $\pi_{X}^{-1}(X_{\sing})$. 
\end{proof}
For a positive integer $j$ with $p\nmid j$, we write $j=np+s$ with
$n\ge0$ and $1\le s\le p-1$. Since $\sht_{V}(j)=D_{V}n+\sht_{V}(s)$,
we have
\begin{align}
\nu(M_{j}) & =l+(p-1-D_{V})n+s-\sht_{V}(s).\label{eq:M_j-3}
\end{align}
Note that since $\sht_{V}(1)=0$, 
\begin{equation}
\nu(M_{1})=l+1>l=\nu(M_{0}).\label{eq:M_1}
\end{equation}
 
\begin{lem}
\label{lem:sht}We have
\[
s-\sht_{V}(s)=\sht_{V}(p-s)+s+d-l-D_{V}.
\]
\end{lem}

\begin{proof}
The proof here is taken from \cite[Proof of Prop. 6.36]{MR3230848}.
We have
\begin{align*}
 & \sht_{V}(p-s)+s+d-l-D_{V}\\
 & =s+\left(\sum_{\lambda=1}^{l}\sum_{i=1}^{d_{\lambda}-1}i+\left\lfloor -\frac{is}{p}\right\rfloor \right)+d-l-D_{V}\\
 & =s+\left(\sum_{\lambda=1}^{l}\sum_{i=1}^{d_{\lambda}-1}-\left\lfloor \frac{is}{p}\right\rfloor -1\right)+d-l\\
 & =s-\sht_{V}(s).
\end{align*}
\end{proof}
\begin{lem}
\label{lem:sht-upper}We have
\[
\sht_{V}(s)\le\frac{(s-1)D_{V}}{p}.
\]
\end{lem}

\begin{proof}
When $i$ varies from 1 to $d_{\lambda}-1$, the differences $\frac{is}{p}-\left\lfloor \frac{is}{p}\right\rfloor $
take $d_{\lambda}-1$ distinct values in $\{1/p,\dots,(p-1)/p\}$.
Therefore 
\[
\sum_{i=1}^{d_{\lambda}-1}\left\lfloor \frac{is}{p}\right\rfloor \le\sum_{i=1}^{d_{\lambda}-1}\frac{is}{p}-\sum_{j=1}^{d_{\lambda}-1}\frac{j}{p}=\frac{s-1}{p}\sum_{i=1}^{d_{\lambda}-1}i=\frac{(s-1)(d_{\lambda}-1)d_{\lambda}}{2p}.
\]
 The lemma follows by taking the sum over $\lambda$. 
\end{proof}
\begin{proof}[Proof of Theorem \ref{thm:precise}]
From Propositions \ref{prop:precise-1} and \ref{prop:M_j}, $\delta(X)=-\infty$
if and only if $\nu(M_{j})$ are not bounded above. In turn, from
(\ref{eq:M_j-3}), this is equivalent to that $D_{V}<p-1$, which
proves the first assertion. Equality (\ref{eq:intro-1}) follows again
from Propositions \ref{prop:precise-1} and \ref{prop:M_j} and from
formulas (\ref{eq:M_j-3}) and (\ref{eq:M_1}). Equality (\ref{eq:intro-2})
follows from Lemma \ref{lem:sht}.
\end{proof}
\begin{proof}[Proof of Theorem \ref{thm:log-term}]
If $\delta(X)=-\infty$, then (\ref{eq:intro-3}) is trivial. We
may suppose that $\delta(X)\ne-\infty$, which is, from Theorem \ref{thm:precise},
equivalent to $D_{V}\ge p-1$. Since 
\[
\sht_{V}(p-(p-1))+(p-1)=p-1,
\]
again from Theorem \ref{thm:precise}, we have 
\[
\delta(X)\le D_{V}-p.
\]

On the other hand, from Theorem \ref{thm:precise} and Lemma \ref{lem:sht-upper},
\begin{align*}
\delta(X) & =D_{V}-1-\max_{1\le s\le p-1}\{\sht_{V}(p-s)+s\}\\
 & \ge D_{V}-1-\max_{1\le s\le p-1}\left\{ \frac{p-s-1}{p}D_{V}+s\right\} \\
 & =\frac{D_{V}}{p}-1-\max_{1\le s\le p-1}\left\{ s\left(1-\frac{D_{V}}{p}\right)\right\} .
\end{align*}
If $D_{V}\ge p$, then 
\[
=\frac{D_{V}}{p}-1-\left(1-\frac{D_{V}}{p}\right)=\frac{2D_{V}}{p}-2.
\]
\end{proof}
\begin{rem}
The additive action of the fixed point part $V^{G}$ on $V$ commutes
with the $G$-action. Therefore $X$ as well as its jet schemes and
the arc space inherit the $V^{G}$-action. Using this structure, we
can show that for a closed subset $C\subset X_{\sing}$, $\nu(M_{j}\cap\pi_{X}^{-1}(C))=\nu(M_{j})-l+\dim C$
with $M_{j}$ as in Proposition \ref{prop:M_j}. By the same argument
as above, we can prove that if $D_{V}\ge p-1$, then
\[
\frac{2D_{V}}{p}-2+l-\dim C\le\discrep{\center\subset C;X}\le D_{V}-p+l-\dim C.
\]
\end{rem}

\bibliographystyle{alpha}
\bibliography{../mybib}

\end{document}